\documentclass[12pt]{amsart}
\usepackage{hyperref}
\usepackage{amsmath, amsthm, amstext}
\usepackage{amssymb, array, amsfonts}
\usepackage[english]{babel}
\usepackage{bbding}
\usepackage{float}
\usepackage{graphics}
\usepackage{graphicx}
   \usepackage{enumerate}
\numberwithin{equation}{section}

\def \er{\varepsilon}

\def \ka{\varkappa}

\def \mut{\tilde{\mu}}

\renewcommand{\l}{\left}
\renewcommand{\r}{\right}
\renewcommand{\mod}{\,\mathrm{mod}\,}

\def \Q{\mathbb{Q}}

\def \N{\mathbb{N}}
\def \M2{\mathrm{M}_2}
\def \R{\mathbb{R}}
\def \Z{\mathbb{Z}}
\def \T{\mathbb{T}}

\def \sl2r{\mathrm{SL}(2,\R)}

\newcommand{\beq}{\begin{equation}}
\newcommand{\eeq}{\end{equation}}

\def\wt{\widetilde}

\def\diag{\operatorname{diag}}

\def\tr{\operatorname{tr}}

\def\wt{\widetilde}

\newcommand{\eqdef}{\stackrel{\rm def}{=\kern-3.6pt=}}

\theoremstyle{plain}
\newtheorem{theorem}{\bf Theorem}[section]
\newtheorem{lemma}[theorem]{\bf Lemma}

\newtheorem{prop}[theorem]{\bf Proposition}
\newtheorem{cor}[theorem]{\bf Corollary}

\theoremstyle{definition}

\theoremstyle{remark}
\newtheorem{remark}[theorem]{\bf Remark}

\theoremstyle{cond}

\renewcommand{\le}{\leqslant}
\renewcommand{\ge}{\geqslant}
\newcommand{\rank}{\mathop{\mathrm{rank}}\nolimits}

\renewcommand{\qed}{\vrule height7pt width5pt depth0pt}
\title[Lipschitz monotone potentials]{All couplings localization for quasiperiodic operators with Lipschitz monotone potentials}
\author{Svetlana Jitomirskaya, Ilya Kachkovskiy}
\date{}
\begin{document}

\begin{abstract}
	We establish Anderson localization for quasiperiodic operator families of the 
	form
	$$
	(H(x)\psi)(m)=\psi(m+1)+\psi(m-1)+\lambda v(x+m\alpha)\psi(m)
	$$
	for all $\lambda>0$ and all Diophantine $\alpha$, provided that $v$ 
	is a $1$-periodic function 
        satisfying a 		Lipschitz
monotonicity condition on $[0,1)$. The localization is uniform on any 
	energy interval on which Lyapunov exponent is bounded from below.
	
\end{abstract}
\maketitle
\section{Introduction}
Ever since the Nobel-prize winning discovery of Anderson that even
weakly coupled 1D structures with random impurities exhibit insulator
behavior (manifested mathematically as Anderson localization) a
paradigm has been that the phenomenon of localization at small
couplings is a signature of ``randomness'',
(see e.g. \cite{damrev}). Localization at large couplings is
intuitive and can be approached perturbatively in a variety of settings. However, small-coupling
localization, i. e. the fact that there is pure point spectrum even for
{\it arbitrary small} perturbations of the Laplacian, is not expected at all in dimensions higher than two, and
even in random 1D  is significantly more subtle than the corresponding high
coupling fact.  Indeed this phenomenon is not
present for analytic quasiperiodic potentials which have purely
absolutely continuous spectrum at small couplings for all phases and
frequencies \cite{eli, bj,aj, avilaarc}, however is expected (yet
apparently difficult to establish) even for the mildly random underlying dynamics
such as skew shifts \cite{fg,kru}. So far however  it was only proved in the random
or quasirandom cases \cite{ckm,bs}. 

In this paper we show that localization at all couplings, that is, 
{\it ``random-like''} behavior,  holds for
all {\it quasiperiodic potentials} that are monotone (in a Lipshitz way) on the period. Such
monotonicity of course  implies discontinuity. It was shown in \cite{DK} (proving
a conjecture made in \cite{mz}) that this already leads to
a.e. positivity of the Lyapunov exponents as discontinuity makes the potentials
non-deterministic in the Kotani sense.  Yet the question of localization
is still very subtle.  First it is not a priori clear if the Lyapunov exponents
are positive on the spectrum (or equivalently whether the spectrum has 
positive measure). Indeed, in the most well studied quasiperiodic
model with discontinuity, the Fibonacci potential (see, for example,
the recent preprint \cite{DGY} for the most comprehensive results and
references) the spectrum  is a Cantor set of zero Lebesgue measure,
and is purely singular continuous for all $\lambda>0$.  The same
holds, more generally, for Sturmian potentials, for all of which
Lyapunov exponents are positive almost everywhere (by the same
discontinuity reason) but not on the zero measure spectrum. Second, even under the condition of  positivity
of the Lyapunov exponents, establishing Anderson localization is a
known difficult problem \cite{bbook}. While monotonicity results in
lack of resonances and thus should lead to
obvious advantages in the perturbative arguments, all the existing
non-perturbative ones \footnote{We employ the word non-perturbative in the widely used by now
  sense of ``obtained as a corollary of positive Lyapunov exponents
  without further largeness/smallness assumptions''
  \cite{bbook,congr}.} have so far required (near) analyticity and used a powerful
analytic apparatus, thus are not applicable here. Indeed the problem
remains extremely difficult even for the  smooth
potentials (e.g. \cite{Sin,eli1,K}) and not much is known beyond
(near) analyticity. In contrast, our results
only require Lipshitz monotonicity and even that can be somewhat relaxed.


More precisely we consider quasiperiodic operator families in $l^2(\Z)$ of the form
\beq
\label{h_def}
(H_{\alpha,\lambda}(x) \psi)(m)=\psi(m+1)+\psi(m-1)+\lambda v(x+m\alpha) \psi(m),\quad m\in \Z,
\eeq
where $v$ is a $1$-periodic function on $\mathbb R$, continuous on $[0,1),$ 
satisfying $v(0)=0$, $v(1-0)=1$, and having the following Lipschitz monotonicity property:
\beq
\label{antilip}
\gamma_-(y-x)\le v(y)-v(x)\le \gamma_+(y-x),\quad \gamma_+,\gamma_->0.
\eeq
Hence, $v$ is strictly increasing on $[0,1)$ and has jump discontinuities 
at integer points. A typical example of such function would be $v(x)=\{x\}$, for which 
we have $\gamma_+=\gamma_-=1$. We show (see Corollary
\ref{localization_cor}) that for  for all $\lambda$  and almost all
$\alpha ,$ for any $v$ satisfying \eqref{antilip} and a.e. $x,$ the spectrum of the operator \eqref{h_def} is purely point and the eigenfunctions decay exponentially at the 
Lyapunov rate.

Additionally, we show that for any operator (\ref{h_def}) satisfying
(\ref{antilip}) the integrated density of states is absolutely
continuous and the Lyapunov exponent is (almost Lipshitz)
continuous. So far, results of this nature have been proved only for
either random (where Wegner's lemma is available) or analytic
quasiperiodic potentials with Diophantine frequencies. In contrast, our
result does not require {\it any} condition on frequency.

Finally, we show that Anderson localization for any operator (\ref{h_def}) satisfying
(\ref{antilip}) is {\it uniform}  on any interval on which Lyapunov
exponent is uniformly positive (such positivity is established in Corollary \ref{lyapunov_cor} for large $\lambda$, and the bound depends only on $v$ for almost every $\alpha$). Uniform
localization, while often considered a feature of localization in
physics literature was shown in \cite{djls,jfa} to not hold for random
or analytic quasiperiodic models as it is incompatible with generic
singular continuous spectrum that occurs in many ergodic  families. In
fact, so far the only known example is in the context of
limit-periodic operators, see \cite{DG}. Operators \eqref{h_def} with
$v$ satisfying \eqref{antilip} provide therefore the first explicit
example of a uniformly localized family.\footnote{Maryland model
  exhibits uniform localization for energies restricted to a finite
  interval, but not overall.} We note that our method covers a wide class of potentials compared to very concrete Fibonacci or Maryland models.


The proof is based on studying  the restrictions of the operator \eqref{h_def} onto intervals of the form $[0,q_k-1]$, where $q_k$ are denominators of the continued fraction approximation of $\alpha,$ with certain properties. It 
is possible to show that the spectra of these restrictions are almost invariant under the 
transformation $x\mapsto x+\alpha$. Due to local monotonicity of the 
eigenvalues as functions of $x$ and using some finite rank perturbation arguments, this leads to linear repulsion of eigenvalues and to a 
Lipschitz bound on the integrated density of states.\footnote{ Even
  though the technical analysis only holds for a sparse sequence of
  scales, this is sufficient for a conclusion on the IDS. This is what
  allows to obtain the result without any Diophantine conditions.} As a consequence, the Lyapunov exponent
is continuous and, for large $\lambda$, is uniformly bounded from below. 
A more careful study of the spectra of $H_{q_k}(x)$ leads to a large deviation theorem. 
The rest of the proof follows the non-perturbative scheme of 
proving Anderson localization in \cite{Lana}, with large deviation
estimate replacing the analytic part of the argument.  Uniformity of localization 
follows from uniformity in the large deviation theorem.

It is interesting how our results compare with the recent work \cite{AK} on 
so-called {\it monotonic cocycles}, i. e. cocycles $A\colon \T^d\to
\sl2r$ for which one can find $w\in \R^d$ such that the map $t\mapsto
\arg\{A(x+tw)\cdot y\}$ has positive derivative in $t$ for all $x\in
\T^d$, $y\in \R^2\setminus\{0\}$. It was discovered that the Lyapunov
exponents of such cocycles behave in a highly regular way; for
example, they are analytic (resp. $C^{\infty}$) in any parameter as
long as the cocycle analytically (resp. $C^{\infty}$) depends on the
parameter. Nothing like that is true for general analytic cocycles:
one can establish continuity in the analytic case \cite{B,BJ2},
but it can be seen that any prescribed continuity modulus is escaped
by a generic set of frequencies and also there is no H\"older continuity
for the critical almost Mathieu operator at some Diophantine
frequencies \cite{hs}. Moreover,  in the $C^{\infty}$ case even the continuity may fail \cite{WY}. The price to pay for regularity is that continuous monotonic cocycles are never homotopic to the identity, and hence there are no examples of monotonic Schr\"odinger cocycles. 
However, for the Schr\"odinger cocycle $S_{v,E}(x)=\begin{pmatrix}
E-v(x)&-1\\1&0
\end{pmatrix}$
one can establish the following identity:
$$
\l\langle
\frac{d}{dx}\l\{S_{v,E}(x+\alpha)S_{v,E(x)}
\begin{pmatrix}u_1\\u_2\end{pmatrix}\r\},
\begin{pmatrix}
0&-1\\
1&0
\end{pmatrix}S_{v,E}(x+\alpha)S_{v,E(x)}
\begin{pmatrix}u_1\\u_2\end{pmatrix}
\r\rangle=
$$
$$
=v'(x)u_1^2+v'(x+\alpha)\l((v(x)-E)u_1+u_2\r)^2>0
$$
whenever $u_1^2+u_2^2>0$. 
This implies that the second iterate of the Schr\"odinger cocycle with
with $v$ satisfying \eqref{antilip} is {\it locally monotonic in $x$}
at the points where $v'(x)$, $v'(x+\alpha)$ exist\footnote{In
  \cite{AK}, monotonicity with respect to a parameter was also
  studied. In particular, it is mentioned that the second iterate of $S_{v,E}$ is monotonic in $E$, rather than the first iterate.}.
In other words, we avoid topological obstruction to monotonicity by
introducing a discontinuous potential. Thus, it may be natural to
expect that some of the properties of monotonic cocycles (such as
regularity of the Lyapunov exponent) survive to some extent in our
case, and the advantage is that we remain in the framework of
Schr\"odinger cocycles. We see that, while our method is completely
different from that of \cite{AK}, regarding continuity (and even
some H\"older continuity), this is indeed the case. Moreover, it
suggests that some general results of \cite{AK} may hold true even  if
discontinuity is allowed.
\section{Preliminaries: density of states and Lyapunov exponent}
Let $\wt H_n(x)$ denote the restriction of
$H_{\alpha,\lambda}(x)$ to $l^2[0,n-1]$ {\it with periodic boundary conditions}. We will 
denote integer intervals by $[0,n-1]$ instead of $[0,n-1]\cap \Z$ where it is clear from the 
context. We have, therefore,
$$
(\wt H_n(x)\psi)(m)=\psi((m+1)\mod\, n)+\psi((m-1)\mod\,n)+\lambda v(x+m\alpha),\quad m\in \{0,\ldots,n-1\}.
$$
The density states measure $N(dE)$ can be defined as the following functional on continuous functions with compact support:
$$
\int_{\R} f(E) N(dE)=\lim_{n\to \infty}\frac{1}{n}\int_{[0,1)}\tr f(\widetilde{H}_{n}(x))\,dx.
$$
The measure $N(dE)$ is a continuous probability measure, and its distribution function is called the {\it integrated density of states} (IDS) and is defined by
$$
N(E):={N}((-\infty,E))={N}((-\infty,E]).
$$
Approximating the characteristic function of $(-\infty,E]$ from above and from below 
by continuous functions, one can easily see that
\beq
\label{ids_def}
N(E)=\lim_{n\to\infty}\frac{1}{n}\int_{[0,1)} \widetilde N_{n}(x,E)\,dx,
\eeq
where
\beq
\label{counting_def}
\widetilde N_{n}(x,E)=\#\sigma(\widetilde H_{n}(x))\cap(-\infty,E]
\eeq
is the counting function of the periodic restriction.

Let $H_n(x)$ be the {\it Dirichlet} restriction 
of $H_{\alpha,\lambda}(x)$ onto $[0,n-1]$, and let $P_n(x,E)=\det(H_n(x)-E)$. The $n$-step transfer matrix is defined by
$$
M_n(x,E):=\prod_{l=(n-1)}^0\begin{pmatrix}E-\lambda v(x+l\alpha)&-1\\1&0 \end{pmatrix}=\begin{pmatrix}P_n(x,E)&-P_{n-1}(x+\alpha,E)\\ P_{n-1}(x,E)&-P_{n-2}(x+\alpha,E)\end{pmatrix},
$$
and the standard definition of the Lyapunov exponent is given by
\beq
\label{gammadef}
\gamma(E)=\lim_{n\to\infty}\frac{1}{n}\int_{[0,1)}\ln\|M_n(x,E)\|\,dx=
\inf_{n\in \N}\frac{1}{n}\int_{[0,1)}\ln\|M_n(x,E)\|\,dx.
\eeq
Thouless formula relates the Lyapunov exponent and the density of states measure:
\beq
\label{thouless}
\gamma(E)=\int_{\R} \ln|E-E'|\,N(dE').
\eeq
The expression \eqref{ids_def} also holds for $N_n$ instead of $\wt N_n$ because $H_{\alpha,\lambda}(x)$ is a rank 2 perturbation of $\wt H_{\alpha,\lambda}(x)$.
\section{Main results}
An irrational frequency $\alpha$ is called {\it Diophantine} if there exist $C,\tau>0$ such that for all $n\in \N$ we have $\|n\alpha\|\ge C|n|^{-\tau}$, where $\|x\|=\min(\{x\},\{1-x\})$. Let 
$$
\er(\alpha)=\liminf\limits_k \frac{q_{k-1}}{q_{k+1}},
$$
where $q_k$ are 
denominators of the continued fraction approximants of $\alpha$. Note that $\er(\alpha)\le \frac12$ for any $\alpha\in\R\setminus \Q$.
\begin{theorem}
\label{density_th}The integrated density of states $N(E)$ of the operator family $H_{\alpha,\lambda}$ is Lipschitz continuous and satisfies
\beq
\label{lipcont}
|N(E')-N(E)|\le \frac{|E'-E|}{\lambda (1-\er(\alpha))\gamma_-}.
\eeq
As a consequence, density of states measure is absolutely continuous, and the 
spectrum of the operator $H_{\alpha,\lambda}(x)$ has positive Lebesgue measure.
\end{theorem}
\begin{cor}
\label{lyapunov_cor}
The Lyapunov exponent of $H_{\alpha,\lambda}$ is continuous in $E$ and the set of zeros of $\gamma(E)$ is a closed subset of zero measure. It also admits a lower bound
\beq
\label{lowerlyap}
\gamma(E)\ge \max\l\{\ln\lambda-\ln \frac{2e}{(1-\er(\alpha))\gamma_-},0\r\}.
\eeq
Hence, $\gamma(E)$ is uniformly positive for large $\lambda$.
\end{cor}

The continuity\footnote{In fact, almost Lipshitz continuity} of $\gamma(E)$ immediately follows from Lipschitz continuity of 
$N(E)$. The fact that the zero set of $\gamma(E)$ has measure zero follows from the 
general result \cite{DK}.

\begin{remark}
Due to \cite[Theorem 29]{Hinchin}, we have $\er(\alpha)=0$ for a full measure set of $\alpha$, with obvious implications for \eqref{lipcont} and \eqref{lowerlyap}. If $v(x)=\{x\}$, then \eqref{lipcont} and \eqref{lowerlyap} hold for {\it all irrational} $\alpha$ with $\er(\alpha)$ replaced by $0$, see Remark \ref{straight}.
\end{remark}

Operator $H$ exhibits {\it uniform  localization} if it has pure point
spectrum and there exist $C,c$ such that for any eigenfunction $\psi$ 
there exists $n_0(\psi)$ so that we have
\begin{equation}\label{psi}
|\psi(n)|\le Ce^{-c|n-n_0|}.
\end{equation}

It is known that for ergodic families $H(x)$ uniform localization implies pure point spectrum for
every $x$ with eigenfunctions satisfying (\ref{psi}) where $C,c$ are uniform in $x$ \cite{djls,jfa,rui}.

Here we introduce a new related notion. We will  say that operator $H$ exhibits {\it uniform  Lyapunov localization} if it has
pure point spectrum with all eigenfunctions decaying exponentially and
at the Lyapunov
rate: for any $\delta>0$, there exists $C(\delta)$ such that for any
eigenfunction $\psi$  satisfying $H \psi=E\psi$, $\|\psi\|_{l^{\infty}}=1$, 
there exists $n_0(\psi)$ so that we have
$$
|\psi(n)|\le C(\delta)e^{-(\gamma(E)-\delta)|n-n_0|}.
$$
The bound is uniform in the sense that there is no dependence on $E$ other than via $\gamma(E)$. Clearly, whenever
the Lyapunov exponent $\gamma$ depends continuously on $E,$  uniform
localization is equivalent to uniform Lyapunov localization plus
nonvanishing of the Lyapunov exponent.\footnote{The direction needed
  for our application  is
  immediate. The other direction  is also  true in
  general for minimal underlying dynamics, because of the uniform upper-semicontinuity.}
Our main theorem is
\begin{theorem}
\label{localization_th}
Let $\alpha$ be Diophantine. Then  for any $\lambda >0,$
$H_{\alpha,\lambda}(x)$ has pure point spectrum, and moreover exhibits
uniform  Lyapunov localization for a.e. $x.$
\end{theorem}
An immediate corollary of Theorem \ref{localization_th}, Theorem \ref{7}, and Corollary
\ref{lyapunov_cor} is 
\begin{cor}\label{localization_cor}

Suppose $\alpha$ is Diophantine and $\lambda >
\frac{2e}{(1-\er(\alpha))\gamma_-}.$ Then $H_{\alpha,\lambda}(x)$ has
  uniform localization for all $x.$

\end{cor}
\begin{remark} It is a very interesting question whether or not
  $\gamma(E)>c(\lambda)>0$ (and therefore whether uniform localization
  holds for all $x$), for {\it all} $\lambda>0.$
\end{remark}
\begin{remark}
Some Diophantine condition in Theorem \ref{localization_th} is
necessary by a Gordon-type argument. We conjecture that the treshold
between pure point and singular continuous spectrum lies at
$\gamma(E)=\beta(\alpha)$ where $\beta(\alpha)=\limsup \frac{\ln
  q_{n+1}}{q_n},$ just like in the almost Mathieu case \cite{ayz,jl}
\end{remark} 
Theorem \ref{density_th} and Corollary \ref{lyapunov_cor} are proved in Section \ref{densityofstates}. Theorem \ref{localization_th} is proved in Sections 
\ref{loc_sect} and \ref{exp_sect}.
\section{Trajectories of irrational rotations}
\label{frequency}
Let $\alpha=[a_0;a_1,\ldots]$, and $\frac{p_k}{q_k}=[a_0;\ldots,a_k]$; note that $q_0=1$. 
We have (see, for example, \cite{Hinchin})
\beq
\label{smallqk}
q_n\alpha-p_n=\frac{(-1)^n}{t_{n+1}q_n+q_{n-1}},\quad \text{where}\quad t_n=[a_n;a_{n+1},\ldots].
\eeq
The following is established in \cite{Raven}.
\begin{prop}
\label{twogap}
Let $k\ge 1$. The points $\{j\alpha\}$, $j=0,\ldots,q_k-1$, split $[0,1)$ to $q_{k-1}$ ``large'' gaps with length $\|(q_k-q_{k-1})\alpha\|$, and $q_k-q_{k-1}$ ``small'' gaps with lengths $\|q_{k-1}\alpha\|$.
\end{prop}
We will also need the following elementary 
two-sided bounds on the lengths of these intervals. 
\begin{prop}
The lengths of the intervals from Proposition $\ref{twogap}$ satisfy
\begin{align}
& \frac{1}{q_k}-\frac{q_{k-1}}{q_k q_{k+1}}\le \|q_{k-1}\alpha\|\le \frac{1}{q_k},\label{lengthsqk1}\\
& \frac{1}{q_k}\le \|(q_k-q_{k-1})\alpha\|\le \frac{1}{q_k}+\frac{1}{q_{k+1}}.\label{lengthsqk2}
\end{align}
\end{prop}
\begin{proof}The upper bound in \eqref{lengthsqk1} follows from \eqref{smallqk}:
$$
\|q_{k-1}\alpha\|=\frac{1}{t_k q_{k-1}+q_{k-2}}\le \frac{1}{q_{k}},
$$
since $t_k=a_k+\frac{1}{t_{k+1}}$. For the lower estimate, we have
$$
\frac{1}{q_k}-\frac{1}{t_k q_{k-1}+q_{k-2}}=\frac{t_k q_{k-1}+q_{k-2}-q_k}{q_k(t_k q_{k-1}+q_{k-2})}=\frac{\frac{q_{k-1}}{t_{k+1}}}{q_k(q_k+\frac{q_{k-1}}{t_{k+1}})}
=\frac{q_{k-1}}{q_k(q_{k+1}+\frac{q_k}{t_{k+2}})}\le \frac{q_{k-1}}{q_k q_{k+1}},
$$
so that
$$
\|q_{k-1}\alpha\|\ge \frac{1}{q_k}-\frac{q_{k-1}}{q_k q_{k+1}}.
$$
As for \eqref{lengthsqk2}, since $\|q_{k-1}\alpha\|\le \frac{1}{q_k}$, we must have 
$\|(q_k-q_{k-1})\alpha\|\ge \frac{1}{q_k}$ because the total length of the intervals is 1. The upper estimate in \eqref{lengthsqk2} follows from
$$
\|(q_k-q_{k-1})\alpha\|\le \frac{1}{t_k q_{k-1}+q_{k-2}}+\frac{1}{t_{k+1} q_{k}+q_{k-1}}\le \frac{1}{q_k}+\frac{1}{q_{k+1}}.\,\,\qedhere
$$
\end{proof}
\begin{lemma}
\label{indifference}
Suppose that $\{x\}\notin(1-\frac{1}{q_{k+1}},1)$ for $k$ even 
and $\{x\}\notin [0,\frac{1}{q_{k+1}})$ for $k$ odd. Then
$$
|\{x+q_k\alpha\}-\{x\}|\le \frac{1}{q_{k+1}}.
$$
\end{lemma}
\begin{proof}
Due to \eqref{smallqk}, we have $\|q_k\alpha\|\le \frac{1}{q_{k+1}}$. The choice of $x$ guarantees that $\{x+q_k\alpha\}$ and $\{x\}$ are both close to 0 or close to 1. Hence,
$$
|\{x+q_k\alpha\}-\{x\}|=\|q_k \alpha\|\le 
\frac{1}{q_{k+1}}.\,\,\qedhere
$$
\end{proof}
{\noindent \bf Good denominators.} Recall that $\er(\alpha)=\liminf\limits_k \frac{q_{k-1}}{q_{k+1}}$. 
For $\er(\alpha)<\er<1$, define 
$$
Q(\alpha,\er)=\l\{q_k\colon \frac{q_{k-1}}{q_{k+1}}\le \er\r\},
$$
For any $\er>\er(\alpha)$, the set $Q(\alpha,\er)$ is infinite.

\section{Lipschitz continuity of the IDS}
\label{densityofstates}
Recall that $\wt H_n(x)$ is the periodic restriction of $H_{\alpha,\lambda}(x)$ onto $l^2[0,n-1]$. 
For a fixed $n$, let $\mut_l(x)$, $0\le l\le n-1$, be the eigenvalues 
of $\wt H_n(x)$ 
in the increasing order, counted with multiplicities. The functions $\mut(x)$ are $1$-periodic and continuous on $[0;1)$ except for the finite set of points $0=\beta_0<\beta_1<\ldots<\beta_{n-1}<1$, where 
\beq
\label{beta_def}
\{\beta_0,\beta_1,\ldots,\beta_{n-1}\}=\{0,\{-\alpha\}, \{-2\alpha\},\ldots,\{-(n-1)\alpha\}\}
\eeq
is the part of the trajectory of the irrational rotation. We have 
\beq
\label{frank}
\rank(\wt H_{n}(\beta_k)-\wt H_{n}(\beta_k-0))=1,\quad \tr(\wt H_{n}(\beta_k)-\wt H_{n}(\beta_k-0))=-\lambda,\quad 0\le k\le n-1,
\eeq
so that all ``jumps'' are negative rank one perturbations caused by discontinuity of $v$ at $1$. Hence, we have
\beq
\label{frank2}
\mut_{l}(\beta_k-0)\le \mut_{l+1}(\beta_k)\le \mut_{l+1}(\beta_k-0),\quad 0\le l \le n-2.
\eeq
From \eqref{antilip}, it follows that the eigenvalues are locally monotonic 
functions of $x$, and we have
\beq
\label{antilip2}
\gamma_-(y-x)\le \tilde \mu_l(y)-\tilde \mu_l(x)\le \gamma_+(y-x), \quad \forall\, x,y\in[\beta_k,\beta_{k+1}).
\eeq
The goal of this section is to study the behavior of these eigenvalues for $n=q_k$ and obtain conclusions for the density of states and the Lyapunov exponent.
\begin{lemma}
\label{almostinvariant}
Suppose that $0\le r\le q_k-1$ and that $x,x-\alpha,\ldots, x-(r-1)\alpha$ satisfy the assumptions of Lemma {\rm\ref{indifference}}. Then
\beq
\label{norm_diff}
|\tilde \mu_m(x)-\tilde \mu_m(x-r\alpha)|\le \frac{\lambda\gamma_+}{q_{k+1}},\quad \text{for}\quad 0\le m \le q_k-1.
\eeq
\end{lemma}
\begin{proof}
Let $\{e_0,\ldots,e_{q_k-1}\}$ be the standard basis in $l^2\{0,1,\ldots,q_k-1\}$, and $T e_l=e_{(l+1)\,\mathrm{mod}\,q_k}$ be the unitary shift operator in this space. We compare the spectra of $\wt H_{q_k}(x)$ and $\wt H_{q_k}(x-r\alpha)$. Let us replace the first operator 
by unitary equivalent $T^r \wt H_{q_k}(x)T^{-r}$. It is easy to see that
$$
T^{r}\widetilde{H}_{q_k}(x) T^{-r}-\wt H_{q_k}(x-r\alpha)=\lambda \diag\{w_0,w_1,\ldots,w_{q_k-1}\},
$$
where
\begin{multline}
w_l=\lambda\l[v(x+((l-r)\mod q_k)\alpha)-v(x+(l-r)\alpha))\r]\\=
\begin{cases} \lambda \l[v(x+(l-r+q_k)\alpha)-v(x+(l-r)\alpha)\r],& 0\le l<r\\
0,&r\le l\le q_k-1.
\end{cases}
\end{multline}
From Lemma \ref{indifference} and the Lipschitz bound on $v$, we get that
$$
\|T^{r}\widetilde{H}_{q_k}(x) T^{-r}-\wt H_{q_k}(x-r\alpha)\|\le \frac{\lambda\gamma_+}{q_{k+1}},
$$
which gives \eqref{norm_diff}.
\end{proof}

\begin{theorem}
\label{kgaps_th}
Suppose that $q_k\in Q(\alpha,\er)$ where $\er(\alpha)<\er<1$. Then, for any $K\le q_k-1$ and $0\le m \le q_k-K-1$, $0\le l\le q_k-1$, we have
\beq
\label{Kgaps}
|\tilde \mu_{m+K}(\beta_l)-\tilde\mu_{m}(\beta_l)|\ge \lambda\l(\frac{K(1-\er)\gamma_-}{q_k}-\frac{3\gamma_+}{q_{k+1}}\r)
\eeq
for $k$ odd and the same with $\beta_l$ replaced by $\beta_l-0$ for $k$ even.
\end{theorem}
\begin{proof}
Start from
considering the case $l+K\le q_k-1$. Due to 
\eqref{frank2} and \eqref{antilip2}, it is easy to see that
$$
\tilde{\mu}_{s+1}(\beta_{j+1})\ge \tilde{\mu}_s(\beta_{j+1}-0)\ge \tilde{\mu}_s(\beta_j)+
\lambda(\beta_{j+1}-\beta_j)\gamma_-,
$$
and thus, by iteration,
$$
\tilde \mu_{m+K}(\beta_{l+K})\ge \tilde \mu_m(\beta_l)+
\lambda(\beta_{l+K}-\beta_l)\gamma_-\ge \tilde \mu_m(\beta_l)+\lambda\frac{K(1-\er)\gamma_-}{q_k},
$$
because $\beta_{l+K}-\beta_l\ge \frac{K(1-\er)}{q_k}$, see Proposition \ref{twogap}. Combining it with
$$
|\tilde \mu_{m+K}(\beta_{l+K})-\tilde \mu_{m+K}(\beta_{l})|\le \frac{\lambda\gamma_-}{q_{k+1}}
$$
from Lemma \ref{almostinvariant}, we get \eqref{Kgaps} (with the coefficient 1 instead of 3 in the last term). The case $l>q_{k}-K-1$ 
follows from the case $l\le q_k-K-1$ also due to Lemma \ref{almostinvariant}, with 
an additional error of $\frac{2\lambda\gamma_+}{q_{k+1}}$.
Depending on whether $k$ is even or odd, 
we apply Lemma \ref{almostinvariant} to the points $\beta_l$ of $\beta_{l}-0$ in order to satisfy the assumptions of Lemma \ref{indifference}.
\end{proof}
\begin{cor}
\label{counting_th}
Let $\wt N_{q_k}(x,E)$ be the counting function \eqref{counting_def}. Suppose that $\er(\alpha)<\er<1$ and $q_k\in Q(\alpha,\er)$. Then, for any $\delta>0$ and $E\le E'$, 
we have
\beq
\label{counting_est}
\wt N_{q_k}(x,E')-\wt N_{q_k}(x,E)\le \frac{(E'-E)q_k}{\lambda(1-\er)\gamma_-} (1+\delta)-C(\delta,\er,\gamma_-,\gamma_+).
\eeq
The same holds for the Dirichlet eigenvalue counting function $N_{q_k}(x,E)$.
\end{cor}
\begin{proof}
For $x=\beta_l$ for $k$ odd and $x=\beta_l-0$ for $k$ even, it is a direct 
consequence of \eqref{Kgaps}, because one can choose $K$ large enough (depending on $\gamma_-$, $\gamma_+$, $\er$, but not $q_k$) and split the eigenvalues from $[E,E']$ 
into clusters of length $K$. The factor $(1+\delta)$ appears because of the second term in 
\eqref{Kgaps}, and we have $\delta\sim\frac{1}{K}$.
If $x\in (\beta_l,\beta_{l+1})$, then, due to monotonicity, 
$$
\wt N_{q_k}(\beta_l-0,E)+1\ge \wt N_{q_k}(\beta_l,E)\ge \wt N_{q_k}(x,E)\ge\wt{N}_{q_k}(\beta_{l+1}-0,E)\ge \wt N_{q_k}(\beta_{l+1},E)-1.
$$
From Lemma \ref{almostinvariant} and Theorem \ref{kgaps_th}, it also follows that
$$
|\wt N_{q_k}(\beta_l,E)-\wt N_{q_k}(\beta_{l+1},E)|\le C(\er,\gamma_-,\gamma_+),
$$
from which \eqref{counting_est} follows. The Dirichlet restriction is a rank
2 perturbation of the periodic restriction, so the claim also holds in that case.
\end{proof}

{\noindent \bf Proof of Theorem \ref{density_th}.} 
The estimate \eqref{lipcont} follows from the definition \eqref{ids_def} 
and Corollary \ref{counting_th}; since it holds for any $\er(\alpha)<\er<1$ and any $\delta>0$, it also 
holds for $\er=\er(\alpha)$ and $\delta=0$. The Lebesgue measure of the spectrum is positive 
because $\sigma(H)$ is the essential support of the 
absolutely continuous measure $N(dE)$.\,\qed

\begin{lemma}
\label{lyapaux}
Let $0\le f(x)\le a$ for all $x\in \mathbb R$, and $\int_\R f(x)\,dx=1$. Then
$$
\int_{\R}f(x)\ln|x| \,dx\ge \frac{1}{a}\int_{-a/2}^{a/2}\ln|x|\,dx=\ln(a/2e).
$$
\end{lemma}
\begin{proof}
By rescaling, we can assume $a=1$. Then
$$
\int_{\R}f(x)\ln|x| \,dx-\ln(1/2e)=\int_{-1/2}^{1/2}(f(x)-1)\ln|x|\,dx+\int_{\R\setminus[-1/2,1/2]}f(x)\ln|x|\,dx\ge
$$
$$
\ge \ln 2 \int_{-1/2}^{1/2}(1-f(x))\,dx-\ln 2 \int_{\R\setminus[-1/2,1/2]}f(x)\,dx=0.\,\,\qedhere
$$
\end{proof}
{\noindent \bf Proof of Corollary \ref{lyapunov_cor}}. The lower bound \eqref{lowerlyap}, which is 
the only thing remaining to prove, immediately follows from Corollary \ref{density_th}, Lemma \ref{lyapaux} and the Thouless formula \eqref{thouless}.\,\qed

\begin{remark}
\label{straight}
In the special case $v(x)=\{x\}$, the eigenvalues $\mu_m(x)$ of $H_n(x)$ are piecewise linear functions of $x$, and we can integrate them explicitly. Denote
$$
P_n(x,E):=\det (H_n(x)-E)=\prod_{l=0}^{n-1} (\mu_l(x)-E).
$$
Using the definition of $\gamma(E)$, we can get the following lower bound.
$$
\gamma(E)\ge \limsup_{n\to\infty} \frac1n \int_0^1\ln|P_n(x,E)|\,dx=\limsup_{n\to\infty} \frac1n \int_0^1\tr\ln\l|H_n(x)-E\r|\,dx
$$
$$
=\limsup_{n\to\infty}\sum_{l=0}^{n-1}\frac{1}{n\lambda}\tr\l[g(H(\beta_{l+1}-0)-E)-g(H(\beta_l)-E)\r],
$$
where $g(\mu)=\mu\ln|\mu|-\mu$ is the antiderivative of $\ln |\mu|$. Regrouping the terms, we that $\gamma(E)$ is bounded from below by
\beq
\label{krein_case}
\limsup_{n\to\infty}\sum_{l=0}^{n-1}\frac{1}{n\lambda}\tr\l[g(H(\beta_l-0)-E)-g(H(\beta_l)-E)\r]=\limsup_{n\to\infty} \frac{1}{n\lambda}\sum_{l=0}^{n-1}\int_{\Sigma_l}\ln|\mu+E|\,d\mu,
\eeq
where $\Sigma_l=\cup_m [\mu_{m}(\beta_l),\mu_m(\beta_l-0)]$
is the support of the difference of counting functions of $H(\beta_l-0)-E$ and $H(\beta_l)-E$. Since $\tr(H(\beta_l-0)-H(\beta_l))=\lambda$, we have $|\Sigma_l|=\lambda$ and, by Lemma \ref{lyapaux}, 
$$
\int_{\Sigma_l}\ln|\mu+E|\,d\mu\ge \int_{-\lambda/2}^{\lambda/2}\ln|\mu|\,d\mu=\lambda(\ln(\lambda/2)-1),
$$
so that
$$
\gamma(E)\ge \max\{0,\ln(\lambda/2e)\}.
$$
The equality \eqref{krein_case} is, in fact, a particular case of Krein spectral shift 
formula for a rank one perturbation.
\end{remark}

\section{Large deviation theorem for $P_{q_k}(x,E)$}
\label{ldt_sect}
Recall that $H_n(x)$ is the Dirichlet restriction of $H_{\alpha,\lambda}(x)$ onto $l^2[0,n-1]$. 
The following two relations are well known and can be easily checked using properties of determinants.
\beq
\label{prel1}
\wt P_n(x,E)+2(-1)^n=P_n(x,E)-P_{n-2}(x+\alpha,E), \quad n\ge 3,
\eeq
\beq
\label{prel2}
P_{n}(x,E)+P_{n-2}(x,E)=(\lambda v(x+(n-1)\alpha)-E)P_{n-1}(x,E),\quad n\ge 2.
\eeq
Here $\wt P_n(x,E)=\det(\wt H_n(x)-E)$, and $P_0(x,E)=1$.
The following result is obtained in \cite{LanaMavi}. It holds for arbitrary $\alpha\in \R\setminus\Q$ and arbitrary piecewise continuous potentials.
\begin{theorem}
\label{upperbound}
For any $\ka>0$ and $E\in \R$ there exists an $N\in \N$ such that $|P_n(x,E)|\le e^{n(\gamma(E)+\ka)}$ for all $n>N$. Moreover, $N$ can be chosen uniformly in $E\in [E_1,E_2]$ as long as $\gamma(E)$ is continuous on this interval.
\end{theorem}
The following large deviation theorem is the main technical part in the proof of localization.
\begin{theorem}
\label{ldt}
Fix $E,\lambda\in \R$ such that $\gamma(E)>0$, and fix $\er(\alpha)<\er<1$.
For any $\delta>0$, there exists $q_0>0$ such that for all $q_k\in Q(\alpha,\er)$, 
$q_k\ge q_0$, we have
\beq
|\{x\in[0,1)\colon|P_{q_k}(x,E)|<e^{q_k(\gamma(E)-\delta)}\}|<
e^{-C\delta q_k},
\eeq
where $c_1$, $c_2$ may also depend on $\gamma_-,\gamma_+,\er,\lambda$, 
but can be chosen uniformly in $E$ on any compact interval. In addition, the set in the left hand side can be covered by at most $q_k$ intervals of size $e^{-C\delta q_k}$.
\end{theorem}
\noindent We need several preparatory lemmas.
\begin{lemma}
\label{qkzeros}
Under the assumptions of Theorem $\ref{ldt}$, the number of zeros of $P_{q_k}(x,E)$ $($counted with multiplicities$)$ is equal to the number of points from $\{\beta_0,\ldots,\beta_{q_k-1}\}$ such that $N_{q_k}(\beta_l-0,E)<N_{q_k}(\beta_l,E)$. The same holds for $\wt P_{q_k}$, $\wt N_{q_k}$.
\end{lemma}
\begin{proof}
Consider the function $N_{q_k}(x,E)$ (or $\wt N_{q_k}$) as $x$ goes from 0 to 1. It decreases by $1$ at each zero 
of $P_{q_k}(x,E)$ (or $\wt P$, respectively) and can increase at most by $1$ at each point $\beta_l$. Since $N_{q_k}(x,E)=N_{q_k}(x+1,E)$, the statement follows.
\end{proof}
\begin{lemma}
\label{goodlemma}
Suppose that $A_1$, $A_2$ are two finite subsets of $[m,M]$ of the same cardinality, 
$m>0$, and that $f$ is a nondecreasing function on $[m,M]$.
Assume that the difference of counting functions of $A_1$ and $A_2$ is bounded by $N$. Then
$$
\l|\sum_{a\in A_1}f(a)-\sum_{a\in A_2}f(a)\r|\le 2N\max\{|f(m)|, |f(M)|\}.
$$
\end{lemma}
\noindent (Note that the values of $f$ may be negative, hence we need to take both $m$ and $M$ into account.)
\begin{proof}
Obviously, the worst case is when $(A_1\setminus A_2)\cup (A_2\setminus A_1)$ contains $2N$ points.
\end{proof}
\begin{lemma}
\label{logthing}
Let $a,b>0$. Then
$$
\sum_{j=1}^n(\ln(aj+b)-\ln (aj))\le \frac{b}{a}\ln (n+1).
$$
\end{lemma}
\begin{proof}
The left hand side is
$$
\sum_{j=1}^n \ln\l(1+\frac{b}{aj}\r)=\ln\prod_{j=1}^n \l(1+\frac{b}{aj}\r)\le \ln \prod_{j=1}^n 
\exp\l(\frac{b}{aj}\r)\le \ln \exp\l\{ \frac{b}{a}\ln (n+1)\r\}\,\qedhere
$$
\end{proof}
\vskip 1mm
{\noindent \bf Proof of Theorem \ref{ldt}.} We study the behavior of the function 
$$
\frac{1}{q_k}\ln|P_{q_k}(x,E)|=\frac{1}{q_k}\sum_{j=0}^{q_k-1}\ln|\mu_j(x)-E|.
$$
In the sequel, all constants are allowed to depend on $\er,\gamma_-,\gamma_+$. 
From Theorem \ref{kgaps_th}, Corollary \ref{counting_th}, and the fact that 
$|N_{q_k}(x,E)-\wt N_{q_k}(x,E)|\le 2$, one can split the eigenvalues $\mu_j(x)$ on each interval into three clusters: above $E$, around $E$, and below $E$, such that
\begin{enumerate}
	\item The eigenvalues $\nu^+_j(x)$ in the cluster above $E$, taken in the {\it increasing} order as $j=1,2,\ldots$, admit a lower bound $\nu_j^+(x)\ge E+j\frac{C_1}{q_k}$.
	\item The eigenvalues $\nu^-_j(x)$ in the cluster below $E$, 
	taken in the {\it decreasing} order as $j=1,2,\ldots$, admit an upper bound $\nu_j^-(x)\le E-j\frac{C_1}{q_k}$.
	\item There are $C_2$ eigenvalues in the remaining cluster around $E$.
	\item The same holds for the eigenvalues of the periodic restriction. We will 
	denote them by $\tilde \nu_j^{\pm}(x)$.
\end{enumerate}
Let us also decompose $P_{q_k}$, $\wt P_{q_k}$ in the same way.
\beq
\label{ppm0}P_{q_k}(x,E)=P_{q_k}^+(x,E)P_{q_k}^0(x,E)P_{q_k}^-(x,E),\quad \wt P_{q_k}(x,E)=\wt P_{q_k}^+(x,E)\wt P_{q_k}^0(x,E)\wt P_{q_k}^-(x,E),
\eeq
where the factors are formed by $(\mu_j(x)-E)$ from the respective clusters. 
We now claim that, for any $x,y\in [0,1)$, we have
\beq
\label{logclose}
\l|\ln|P^{\pm}_{q_k}(x,E)|-\ln|P^{\pm}_{q_k}(y,E)|\r|\le C\ln q_k,\quad 
\l|\ln|\wt P^{\pm}_{q_k}(x,E)|-\ln|\wt P^{\pm}_{q_k}(y,E)|\r|\le C\ln q_k.
\eeq
We start from proving \eqref{logclose} for $x=\beta_l$, $y=\beta_m$. 
Due to Lemma \ref{almostinvariant}, we have $|\tilde{\nu}_j^{\pm}(\beta_l)-\tilde{\nu}_j^{\pm}(\beta_m)|\le \frac{\lambda\gamma_+}{q_{k+1}}$, from which it follows that 
$$
\l|\ln\l|\wt P^{\pm}_{q_k}(\beta_l,E)\r|-\ln\l|\wt P^{\pm}_{q_k}(\beta_m,E)\r|\r|\le \sum_{j=1}^{q_k}\l\{\ln\l(\frac{j C_1}{q_k}+\frac{\lambda\gamma_+}{q_{k+1}}\r)-
\ln \frac{j C_1}{q_k}\r\}\le C\ln q_k.
$$
by Lemma \ref{logthing}.
Let us now consider the case $[x,y)\in [\beta_l,\beta_{l+1})$. Due to monotonicity, we can 
assume that $x=\beta_l$, $y=\beta_{l+1}-0$, and then the statement also follows from
$$
\l|\ln\l|\wt P^{\pm}_{q_k}(\beta_l,E)\r|-\ln\l|\wt P^{\pm}_{q_k}(\beta_{l+1}-0,E)\r|\r|\le
\sum_{j=1}^{q_k}\l\{\ln\l(\frac{j C_1}{q_k}+\frac{\lambda\gamma_+}{(1-\er)q_{k}}\r)-
\ln\frac{j C_1}{q_k}\r\}\le C\ln q_k
$$
by Lemma \ref{logthing}. The claim also holds for $P^{\pm}$ by Lemma \ref{goodlemma}.

The above computations imply that there exists $\gamma_{q_k}(E)$ such that
\beq
\label{gammaqk1}
\gamma_{q_k}(E)\le \frac{1}{q_k}\ln|P_{q_k}^-(x,E)P_{q_k}^+(x,E)|\le \gamma_{q_k}(E)+C\frac{\ln q_k}{q_k},
\eeq
\beq
\label{gammaqk2}
\gamma_{q_k}(E)\le \frac{1}{q_k}\ln|\wt P_{q_k}^-(x,E)\wt P_{q_k}^+(x,E)|\le \gamma_{q_k}(E)+C\frac{\ln q_k}{q_k}.
\eeq
This means that, if $|P_{q_k}(x,E)|<e^{q_k(\gamma_{q_k}(E)-\delta)}$, we must have 
$|P_0(x,E)|\le e^{-q_k \delta}$, and hence for some $l$ we have $|\mu_l(x)-E|\le e^{-C q_k \delta}$, which implies that $x$ is either exponentially close to a zero of $P_{q_k}(\cdot,E)$, or to $\beta_l$ for some $l$. In the second case, if $x$ is not close to an ``actual'' 
root of $P_{q_k}$, then $N_{q_k}(\beta_l-0,E)=N_{q_k}(\beta_l,E)$. Hence, by Lemma \ref{qkzeros}, we can add all these $\beta_l$ to the set of roots 
and still get a set with at most $q_k$ points to which $x$ must be exponentially close.
We thus have verified the statement of the theorem, but for $\gamma(E)$ replaced by $\gamma_{q_k}(E)$.

We now claim that $\gamma_{q_k}(E)=\gamma(E)+o(1)$ as $q_k\to \infty$, uniformly in $E$. 
Fix $\ka>0$. From Theorem \ref{upperbound}, we get that, for $n>N(\ka)$, $|P_n(x,E)|\le e^{n(\gamma(E)+\ka)}$. From the definition \eqref{gammadef} of $\gamma(E)$, it follows that, for all $k$, we must have $|P_n(x,E)|\ge C e^{n (\gamma(E)-\ka)}$ for $n=q_k,q_k-1$ or $q_k-2$, on a subset of $[0,1]$ of measure at least $1/4$. If $n=q_k-1$, then \eqref{prel2}
implies that it should hold for $P_{q_k}$ or $P_{q_k-2}$ on a set of sufficiently large measure (bounded from below by positive universal constant). Finally, if it holds for $P_{q_k-2}$, then \eqref{prel1} implies the similar statement for $P_{q_k}$ or $\wt P_{q_k}$, and the case $\wt P_{q_k}$ implies the case of $P_{q_k}$ because of \eqref{gammaqk1}, \eqref{gammaqk2}; thus, Theorem \ref{ldt} follows.\,\,\qed

\section{Localization}
\label{loc_sect}
A {\it generalized eigenfunction} of $H$ is, by definition, 
a polynomially bounded solution of the equation
$H\psi=E\psi$. The corresponding $E$ is called a
{\it generalized eigenvalue}. We first prove
\begin{theorem}\label{7}
\label{main}
Suppose that $\alpha$ is Diophantine, $E$ is a generalized eigenvalue of $H_{\alpha,\lambda}(x)$, and that $\gamma(E)>0$. Then the corresponding generalized eigenfunction belongs to $l^2(\Z)$.
\end{theorem}
From now on, let us drop the dependence on $\alpha$ and $\lambda$ from all the 
notation, assuming that they are fixed. By $G_{[a,b]}(x;m,n)$ we denote the 
$(m,n)$-matrix element of $\l(\l.(H(x)-E)\r|_{[a,b]}\r)^{-1}$ with 
Dirichlet boundary conditions. Note that
$$
G_{[a,b]}(x+k\alpha;m,n)=G_{[a+k,b+k]}(x;m+k,n+k),
$$
so it is sufficient to consider the intervals $[0,n]$. 

Assume that $x$ is fixed. Following \cite{Lana}, let 
us call a point $m\in \Z$ {\it $(\mu,q)$-regular} if there exists an interval $[n_1,n_2]$ such that $n_2=n_1+q-1$, $m\in[n_1,n_2]$, $|m-n_i|\ge q/5$ for $i=1,2$, and
$$
|G_{[n_1,n_2]}(x;m,n_i)|<e^{-\mu |m-n_i|}.
$$
Otherwise, $m$ is called {\it $(\mu,q)$-singular}. Any formal solution $H(x)\psi=E\psi$ 
can be reconstructed from its values at two points,
\beq
\label{expansion}
\psi(m)=-G_{[n_1,n_2]}(x;m,n_1)\psi(n_1-1)-G_{[n_1,n_2]}(x;m,n_2)\psi(n_2+1),\quad m\in [n_1,n_2].
\eeq
If $\mu$ is fixed, then any point $m$ such that $\psi(m)\neq 0$ is $(\mu,q)$-singular for sufficiently large $q$.

\begin{theorem}
\label{multiscale}
Under the assumptions of Theorem {\rm\ref{main}}, let  $\er(\alpha)<\er<1$. 
For any $0<\delta<\gamma(E)$ there exists $q_0$ such that if $q_k>q_0$, $q_k\in Q(\alpha,\er)$, and
$n,m$ are both $(\gamma(E)-\delta,q_k)$-singular with $|m-n|>\frac{q_k+1}{2}$, then $|m-n|>e^{C(\alpha,\er)q_k}$.
\end{theorem}
\begin{proof}
The proof follows the scheme from \cite{Lana}. We have the following expressions for Green's function matrix elements if $b=a+q_k-1$, $a\le l\le b$.
\beq
\label{greenleft}
|G_{[a,b]}(x;a,l)|=\l|\frac{P_{b-l}(x+(l+1)\alpha)}{P_{q_k}(x+a\alpha)}\r|,
\eeq
\beq
\label{greenright}
|G_{[a,b]}(x;l,b)|=\l|\frac{P_{l-a}(x+a \alpha)}{P_{q_k}(x+a\alpha)}\r|.
\eeq
Suppose that $m-[3q_k/4]\le l\le m-[3q_k/4]+[(q_k+1)/2]$. Since $m$ is $(\gamma(E)-\delta,q_k)$-singular, we either have 
\beq
\label{greenfunctionlarge}
|G_{[a,b]}(x;a,l)|>e^{-(l-a)(\gamma(E)-\delta)}\quad \text{or}\quad  |G_{[a,b]}(x;l,b)|>e^{-(b-l)(\gamma(E)-\delta)}
\eeq
for all intervals $[a,b]$ such that $|a-l|,|b-l|\ge q_k/5$ and $b=a+q_k-1$. From Theorem \ref{upperbound} and since $q_k\ge q_0$, we can choose a 
sufficiently large $q_0$ (depending only on $\delta$) such that
$$
|P_{b-l}(x+(l+1)\alpha)|\le e^{(b-l)(\gamma(E)+\delta/32)},\quad |P_{l-a}(x+a\alpha)|\le e^{(l-a)(\gamma(E)+\delta/32)}.
$$
In other words, the numerators of \eqref{greenleft}, \eqref{greenright} cannot get very large. Hence, the only possibility for \eqref{greenfunctionlarge} is for one of the denominators 
to become exponentially small. This means that if $m-[3q_k/4]\le a\le m-[3q_k/4]+[(q_k+1)/2]$, we have (without loss of generality, in the case of the first denominator) 
$$
|P_{q_k}(x+a\alpha,E)|\le \frac{e^{(b-l)(\gamma(E)+\delta/32)}}{e^{-(l-a)(\gamma(E)-\delta)}}=e^{\gamma(E)(b-a)+(b-l)\delta/32-(l-a)\delta}\le e^{q_k(\gamma(E)-\delta/16)}.
$$ Suppose that the points $m_1$ and $m_2=m_1+r$
are both $(\gamma(E)-\delta,q_k)$-singular, $r>0$. Let 
$$
x_j=\{x+(m_1-[3q_k/4]+(q_k-1)/2+j)\alpha\},\quad j=0,\ldots, [(q_k+1)/2]-1,
$$
$$
x_j=\{x+(m_2-[3q_k/4]+(q_k-1)/2+j-[(q_k+1)/2])\alpha\},\quad j=[(q_k+1)/2],\ldots, q_k.
$$
If $r>\frac{q_k+1}{2}$, then all these points are distinct, and we have $|P_{q_k}(x_j,E)|\le e^{q(\gamma(E)-\delta/16)}$. From Theorem \ref{ldt}, we get that, for sufficiently large $q_k$, at least 
two of the points should be $e^{-C q_k}$-close to each other, and so we get that $\|r'\alpha\|\le e^{-Cq_k}$ for some $r'\le r$. From Diophantine condition, we get that $r\ge e^{Cq_k/\tau}$. This completes the proof.
\end{proof}

{\noindent \bf Proof of Theorem \ref{7}.} Suppose that $\psi$ is a generalized eigenfunction, so that $|\psi(m)|\le C(1+|m|^p)$. Fix $\er=1/2$, then $Q(\alpha,\er)$ consists of all denominators of $\alpha$.
Without loss of generality, we may assume that $\psi(0)\neq 0$. Fix $0<\delta<\gamma(E)$. 
The point $0$ is $(\gamma(E)-\delta,q_k)$-singular for sufficiently large $q_k\in Q(\alpha,\er)$. Hence, the interval $[q_k,e^{C(\alpha)q_k}]$ contains only $(\gamma(E)-\delta,q_k)$-regular points.

Take $n\in \N$, and find $k$ such that $n\in [q_k,q_{k+1})$. Since $\alpha$ is 
Diophantine, we have $q_{k+1}\le q_k^{C'(\alpha)}$. Together with the previous observation, 
if $n$ is sufficiently large, it is contained in an interval $[q_k,q_k^{C'(\alpha)})$ 
consisting of $(\gamma(E)-\delta,q_k)$-regular points.
Hence there exist
$n_1,n_2$ satisfying $n_1\le n\le n_2$ and $q_k/5 \le |n_2-n_1|\le 4q_k/5$, 
such that
$$
G_{[n_1,n_2]}(x;n,n_i)\le e^{-(\gamma(E)-\delta)|n-n_i|}.
$$
From \eqref{expansion}, we obtain
$$
|\psi(n)|\le C(1+|n|^p)e^{-\frac{\gamma(E)-\delta}{5}q_k}\le C(1+|n|^p)e^{-C_1 n^{1/C'(\alpha)}},
$$
which holds for sufficiently large $n$. The case $n<0$ is similar, and thus $\psi\in l^2(\Z)$.
\section{Proof of Theorem \ref{localization_th}: exponential decay of eigenfunctions}
\label{exp_sect}
 Due to \cite[Chapter VII]{Berez}, the spectral measure of $H_{\alpha,\lambda}(x)$
 is supported on the set of its generalized eigenvalues.
 The zero set of $\gamma(E)$ has Lebesgue and density of states measure zero due to Theorem \ref{density_th} and Corollary \ref{lyapunov_cor}. 
 Hence, for almost 
 every $x$, the set of energies $E$ for which the statement of Theorem
 \ref{7} holds has full spectral measure, implying the pure point spectrum. Thus it remains to prove
 uniform Lyapunov localization for this full measure set of $x.$
If $H_{\alpha,\lambda}(x)\psi=E\psi$, $\psi\in l^2(\Z)$, let $n_0(\psi)$ be the 
leftmost point where $|\psi(n)|$ attains its maximal value (which obviously exists). Then we assume that $\psi(n_0)=1$. Suppose that $\alpha$ is Diophantine and that $\gamma(E)>0$. Our goal is to show that, if $0<\delta<\gamma(E)$, then
$$
|\psi(n)|\le C(\delta)e^{-(\gamma(E)-\delta)|n-n_0(\psi)|},
$$
where the constant $C$ does not depend on $E$ and $x$.

Fix any $\delta<\gamma(E)$. Similarly to the previous section, let us also fix $\er=1/2$, so that $Q(\alpha,1/2)$ contains all denominators of $\alpha$.

Without loss of generality, we can assume that $n_0(\psi)=0$, otherwise we can shift $x$ to $x+n_0(\psi)\alpha$ and get a unitary equivalent operator whose eigenfunctions are translated by $n_0(\psi)$. There exists $q_0(\delta)>0$ such that $0$ is $(\gamma(E)-\delta/2,q_k)$-singular for all denominators $q_k>q_0$, and that the statement of Theorem \ref{multiscale} holds for $q_0$ and $\delta/2$. Note that this choice is uniform in 
$x$ and $E$ as $E$ must belong to $\sigma(H_{\alpha,\lambda}(x))$ which is contained in a uniformly bounded interval.

The rest of the proof follows the method of \cite{Lana}. If $q_k>q_0$, then $[q_k,e^{c(\alpha)q_k}]$ must consist of $(\gamma(E)-\delta/2,q_k)$-regular points. 
Again, from the Diophantine condition, 
the intervals $[q_k,q_k^{C_1(\alpha)}]$ cover all sufficiently large integer points. 
For some $\theta>0$, it also holds for the intervals $[q^{1+\theta}_k,q_k^{C_1(\alpha)}]$, and they also consist of $(\gamma(E)-\delta/2,q_k)$-regular points. Let $n\in [q^{1+\theta}_k,q_k^{C_1(\alpha)}]$. The fact that $n$ is 
$(\gamma(E)-\delta/2,q_k)$-regular implies existence of a certain
interval $[n_1,n_2]$ with Green function's decay from $n$ to the edges of the interval. 
Apply \eqref{expansion} on this interval, thus expanding $\psi(n)$ in terms of $\psi(n_1-1)$ and $\psi(n_2+1)$. The points $n_1-1$ and $n_2+1$ are also regular, and hence we can repeat 
the procedure and get an expansion involving the values of $\psi$ at four points. 
Let us repeat the procedure of finding a suitable interval and expressing each $\psi(n)$ 
using \eqref{expansion} until we get $n_1<q_k$ at some stage, or the depth of the expansion reaches $[5n/q_k]$, whichever comes first. The last condition guarantees that $n_2$ will 
never be singular, as the maximal possible value of $n_2$ on the last 
step does not exceed $2n$, since $n_2\le n+4q_k/5$ on each step.

The result can be written in the following form:
\beq
\label{greenexp}
\psi(n)=\sum_{s\in S}G_{1,s} G_{2,s}\ldots G_{p(s),s}\psi(n_s),
\eeq
where $S\subset\bigcup_{p=0}^{[5n/q_k] +1}\{0,1\}^p$ indicates all possible sequences of choices between $n_1$ and $n_2$ in applying \eqref{expansion} at each step. There are at most $2^{[5n/q_k]+1}$ terms, and at each term we either have $n_s<q_k$ or $p(s)\ge [5n/q_k]+1$. $G_{i,s}$ are matrix elements of $G$ appearing in \eqref{expansion}. Note 
that the only way to reach $(\gamma(E)-\delta/2,q_k)$-singular point in this construction is to get $n_s<q_k$ in which case the process stops and we no longer need to apply \eqref{expansion}.

Let us first consider the case $n_s<q_k$. Using the definition of regularity, 
we have
$$
|G_{1,s} G_{2,s}\ldots G_{p(s),s}|\le e^{-(\gamma(E)-\delta/2)(n-n_s)},
$$
and so
$$
|G_{1,s} G_{2,s}\ldots G_{p(s),s}\psi(n_s)|\le e^{-(\gamma(E)-\delta/2)(n-q_k)}\le 
e^{-n(\gamma(E)-\delta/2)(1-q_k^{-\theta})}\le e^{-(\gamma(E)-\delta/2-\delta_1)n},
$$
where $\delta_1$ can be made smaller than, say, $\delta/3$ 
by choosing a sufficiently large $q_k$. Note that 
this choice would not be uniform if $\gamma(E)$ could become uncontrollably large, but, since 
$\lambda$ and $\alpha$ are fixed, $\gamma(E)$ is continuous, and $E$ belongs to the 
spectrum, this is not the case.

Let us now assume that the number of factors is at least $[5n/q_k]+1$. Then we can use the facts that $|G_{i,s}|\le e^{-(\gamma(E)-\delta/2)q_k/5}$ and $|\psi(n_s)|\le 1$, 
and obtain
$$
|G_{1,s} G_{2,s}\ldots G_{p(s),s}\psi(n_s)|\le e^{-(\gamma(E)-\delta/2)\frac{5n}{q_k}\frac{q_k}{5}}\le e^{-(\gamma(E)-\delta/2)n}.
$$
Combining everything into \eqref{greenexp}, we get
$$
\psi(n)\le 2^{[5n/q_k] +1} e^{-(\gamma(E)-\delta/2-\delta_1)n}\le e^{-(\gamma(E)-\delta)n}
$$
for sufficiently large $q_0$ and $n>q_0$.\qed
\section{Acknowledgements} S.J. is a 2014--15 Simons Fellow. This research was partially
    supported by the NSF DMS--1401204. I.K. was supported by the AMS--Simons Travel Grant 2014--16. 
    We are also grateful to the Isaac
    Newton Institute for Mathematical Sciences, Cambridge, for support
    and hospitality during the programme Periodic and Ergodic Spectral
    Problems where a part of this work was done.

\end{document}